\documentclass[12pt,reqno]{amsart}
\usepackage{amsmath,amsthm,amsfonts,amssymb,amscd,amstext}
\usepackage[latin1]{inputenc}
\usepackage[dvips]{graphicx}
\usepackage{psfrag}
\usepackage{euscript}
\usepackage{a4wide}

\numberwithin{equation}{section}

\usepackage[colorlinks=true,linkcolor=red,urlcolor=black]{hyperref}

\theoremstyle{plain}
\newtheorem{maintheorem}{Theorem}

\newtheorem{maincorollary}[maintheorem]{Corollary}

\newtheorem{theorem}{Theorem}[section]

\newtheorem{corollary}[theorem]{Corollary}
\newtheorem{lemma}[theorem]{Lemma}

\theoremstyle{definition}
\newtheorem{remark}[theorem]{Remark}

\renewcommand{\epsilon}{\varepsilon}

\newcommand{\Fix}{\operatorname{Fix}}

\newcommand{\intt}{\operatorname{int}}

\begin{document}

\title[A Kingman-like Theorem]
{A Kingman-like Theorem}

\thanks{ L.S. is partially supported Fapesb-JCB0053/2013, CNPq and PRODOC-UFBA/2014. V.C. is supported by CAPES.  }

\renewcommand{\subjclassname}{\textup{2020} Mathematics
  Subject Classification}

\author{Vinicius Coelho}

\author{Luciana Salgado}

\address[L.S.]{Universidade Federal do Rio de Janeiro, Instituto de
   Matem\'atica\\
   Avenida Athos da Silveira Ramos 149 Cidade Universit\'aria, P.O. Box 68530,
   21941-909 Rio de Janeiro-RJ-Brazil }
 \email{lsalgado@im.ufjr.br, lucianasalgado@ufrj.br}

\address[V.C.]{Universidade Federal do Oeste da Bahia, Centro Multidisciplinar de Bom Jesus da Lapa\\
Av. Manoel Novais, 1064, Centro, 47600-000 - Bom Jesus da Lapa-BA-Brazil}
\email{viniciuscs@ufob.edu.br}

\begin{abstract}
We provide a Kingman-like Theorem for arbitrary finite measures and a version of Birkhoff's Theorem for bounded  observable. As an application, we show that Birkhoff's limit exists for some continuous observable, in an example of Bowen.
\end{abstract}

\date{\today}
\maketitle
\tableofcontents

\section{Introduction}
\label{sec:intro}

Let $(M,\mathcal{A},\mu)$ be a measure space equipped with a $\sigma$-finite measure, and $T:M\to M$  be a measurable map.

If $\mu(A) = \mu(T^{-1}(A))$ for all $A \in \mathcal{A}$ then $\mu$ is said to be \emph{invariant} under $T$ or, equivalently, $T$ is \emph{measure-preserving}.

Two of the most important results of invariant measures theory are  Kingman's Theorem (see \cite{AB09}) and Birkhoff's Theorem (see \cite{BIR31}).

The basic idea to proof Kingman's Theorem is to apply Fekete's Subadditive Lemma. This Lemma yields information about subadditive sequences $(a_{n})_{n}$ in $\mathbb{R}$ proving that the limit $\lim\limits_{n \to \infty} \frac{a_{n}}{n} = \inf \limits_{n}\frac{a_{n}}{n} =a $  and satisfies $-\infty \leq a < \infty$. This sequence occurs naturaly when we deal with invariant measures and subbaditive sequences of functions for a transformation in a manifold.

Derriennic \cite{D83} generalized Fekete's Lemma as follows. Let  $(a_{n})_{n}$ be a sequence in $\mathbb{R}$ and $(c_{n})_{n}$ be a sequence  such that $c_{n}\geq 0$. If $a_{n+m} \leq a_{n} +a_{m} + c_{n}$ for all $n,m \geq 1$, and $\lim\limits_{n} \frac{c_{n}}{n}=0$ then the limit $\lim\limits_{n} \frac{a_{n}}{n} = a$ and satisfies $-\infty \leq a < \infty$.  He utilizes this result and others techniques to provide a generalization for Kingman's Theorem.

Other  generalisations of Kingman's Theorem were proved by Akcoglu and Sucheston \cite{AS78} (for superadditive processes), Shurger \cite{S91} (a stochastic analogue of generalization of Kingman's Theorem given by Derriennic), and  recently by  A. Karlsson and Margulis \cite{KM99} (for ergodic measure preserving transformations).

Here, we will show a Kingman-like Theorem for an arbitrary finite measure  assuming some conditions. This theorem was inspired by the proof of Kingman's Theorem given by Avila and Bochi \cite{AB09}.

Generalisations of Birkhoff's Theorem were proved by E.Hopf \cite{Ho37} (for infinite measure preserving transformations), J. Aaronson \cite[Theorem 2.4.2]{AA97} (for conservative ergodic measure transformations),  W. Hurewicz \cite{Hu44} (for conservative nonsingular transformations where the observables are defined by means of Radon-Nykodim Theorem and the measure can be finite or infinite),  R. Chacon, D.Ornsten \cite{ChaOr60} (for Markov operators),  M. Carvalho and F. Moreira \cite{CM12} (for half-invariant measures),  M. Carvalho and F. Moreira \cite{CM18} (for ultralimits by means of ultrafilters), and recently M. Lenci and S. Munday \cite{LM18} (for conservative, ergodic, infinite-measure preserving dynamical systems) and, in the context of random walks, D. DolgoPyat, M. Lenci and P. N\'andori \cite{DLN20} proved strong laws of a large number of global observables.

As a consequence of our Kingman-like Theorem, we formulated  a version of  Birkhoff's Theorem for bounded observables and finite measures. Our result are not contemplated by previous work:

($a$) in \cite{Hu44}, Hurewicz worked in context of conservative transformations and bounded  observables  defined by means of Radon-Nykodim Theorem;

($b$) in \cite[Theorem 1.2]{CM12}, Carvalho and Moreira showed that every finite and half-invariant measure is an invariant measure, and our theorem was proved for a finite arbitraty measure;

($c$) in \cite{CM18}, Carvalho and Moreira showed that the Birkhoff's Theorem holds for each non-principal ultrafilter, so for this Theorem to imply our result it is necessary that the value of integral be the same for each non-principal ultrafilter, however it is not clear how to compute this, because the ultrafilters are obtained by Zorn's Lemma, and therefore we do not have an expression for these ultrafilters.

An interesting consequence of our result is the following.  Let $X:M\times \mathbb{R} \to M$ be a continuous flow, and $M$ to be a compact metric space. Consider $X_{t}:M\to M$ given by $X_{t}(x)= X(t,x)$, and $f_{t}:M \to M$ defined by $f_{t}=X_{t}$.  Suppose that $\varphi:M \to \mathbb{R}$ is a continuous function, and fix $x \in M$. If the following inequality 
\begin{align} \limsup\limits_{n} \frac{1}{n} \int_{0}^{n} \varphi \circ f_{t} (y_{x}) dt  \leq    \liminf\limits_{n}  \frac{1}{n} \int_{0}^{n} \varphi \circ f_{t} (x) dt 
\end{align} 
holds for all $y_{x} \in \omega(x)$, then the limit  $\lim \limits_{T\to \infty}  \frac{1}{T}  \int_{0}^{T} \varphi \circ f_{t}(x) dt$  exists. 

We use this to show that for some continuous observables the Birkhoff's limit exists in an example of Bowen.

\section{Statements of main results} \label{sec:statement-result1}

First of all, we introduce some definitions and notations that will be appear on text. Let $(\varphi_{n})_{n}$ be a sequence of  measurable functions where  $\varphi_{n}: M \to \mathbb{R}$ for each $n$ in $\mathbb{N}$. We say that $(\varphi_{n})_{n}$ is a subadditive sequence for $f$ if $\varphi_{m+n} \leq \varphi_{m} + \varphi_{n} \circ f^{m}$ for all $m,n \geq 1$.

We consider a function  $\varphi_{-}: M \to [-\infty, \infty]$  given by $\varphi_{-}(x) = \liminf\limits_{n} \frac{\varphi_{n}(x)}{n}$. For each  $\varepsilon >0$ fixed and $k \in \mathbb{N}$ we define

\begin{center}
$E_{k}^{\varepsilon}=\{x \in M:\varphi_{j}(x) \leq j (\varphi_{-}(x)+\varepsilon)$ for some  $j \in \{1,...,k\}\}$.
\end{center}

Note that $E^{\varepsilon}_{k} \subseteq E^{\varepsilon}_{k+1}$ and  $M= \bigcup\limits_{k=1}^{\infty} E^{\varepsilon}_{k}$.

\begin{maintheorem}\label{mtheo0a}
Let $(M, \mathcal{A}, \mu )$  be a measure space, $f: M \to M$   be a measurable function,  $\mu$ be a finite measure. Suppose that  $(\varphi_{n})_{n}$ is a subadditive sequence for $f$ such that $\varphi_{1} \leq \beta$ for some $\beta \in \mathbb{R}$. If the following conditions are satisfied:
\begin{itemize}

\item[($a$)] for all $j \in \mathbb{N}$ we have that $\varphi_{-}(f^{j}(x)) = \varphi_{-}(x)$  $\mu-$almost everywhere $x$ in $M$;

\item[($b$)]  $\lim\limits_{k \to \infty}\limsup\limits_{n }\frac{1}{n}\sum\limits_{i=0}^{n-k-1}  \mu(f^{-i}(M\setminus E_{k}^{\frac{1}{\ell}})) =0$ for each $\ell \in \mathbb{N} \setminus \{0\}$.

\end{itemize}

Then   $ \int \varphi_{-} d\mu =\inf\limits_{n}   \frac{1}{n}\int \varphi_{n} d\mu $. Moreover, if there exists $\gamma >0$ such that for all $n >0$, $\frac{\varphi_{n}}{n} \geq -\gamma $ then $\int \varphi_{-} d\mu  = \inf\limits_{n}   \frac{1}{n}\int \varphi_{n} d\mu = \lim \limits_{n}   \frac{1}{n}\int \varphi_{n} d\mu $.

\end{maintheorem}

Our goal is to provide a Kingman-like Theorem for an arbitrary measure assuming only the conditions $(a)$ and $(b)$. Moreover, we obtain the convergence of integrals even without a subadditive sequence of real numbers given by Fekete's Lemma (or same version of this result) as is usual when we work with invariant measures.

Let $(M, \mathcal{A}, \mu )$  be a measure space, $f: M \to M$  be a measurable transformation,  $\mu$  be a probability measure. Let $\varphi: M \to \mathbb{R}$ be a  measurable function, we consider $(\varphi_{n})_{n}$ the additive sequence for $f$  given by $\varphi_{n} := \sum\limits_{j=0}^{n-1} \varphi \circ f^{j}$ for each $n$ in $\mathbb{N}$,  and  $\varphi_{-},\varphi_{+}$ the functions defined from $M$ to $[-\infty, \infty]$ given by $\varphi_{-}(x) = \liminf\limits_{n} \frac{\varphi_{n}(x)}{n}$, and $\varphi_{+}(x) = \limsup\limits_{n} \frac{\varphi_{n}(x)}{n}$.

Note that for every bounded function $\varphi: M \to \mathbb{R}$ we  have  that  $\varphi_{-}(f^{j}(x)) = \varphi_{-}(x)$  $\mu-$almost everywhere $x$ in $M$ for      all $j \in \mathbb{N}$.

\begin{remark}
Under the the same hypotheses of Theorem \ref{mtheo0a} with condition $(b)$ replaced by condition $(c)$, that says
\begin{itemize}
\item[$(c)$]$ \mu(f^{-i}(M\setminus E_{k}^{\varepsilon})) \leq \mu(M\setminus E_{k}^{\varepsilon}) $ for all $i \in \mathbb{N}$, for any $k\in \mathbb{N}$, and $\varepsilon>0$,
\end{itemize}

we obtain the conclusion of Theorem \ref{mtheo0a} since the condition $(c)$ implies the condition $(b)$.
\end{remark}

We say that an observable $\varphi$ satisfies hypothesis $(c)$ if for all $i,k \in \mathbb{N}$ and $\varepsilon>0$, the following inequality $ \mu(f^{-i}(M\setminus E_{k}^{\varepsilon})) \leq \mu(M\setminus E_{k}^{\varepsilon}) $  holds when we consider $(\varphi_{n})_{n}$ an additive sequence for $f$.  We observe that if the  measure $\mu$ is an invariant measure, then every observable satisfies  hypothesis $(c)$.

Since every bounded observable satisfies hypothesis $(a)$, we deduce  Birkhoff's Theorem  for finite measures and bounded observables  as follows.

\begin{maincorollary}\label{corA}
Let $(M, \mathcal{A}, \mu )$  be a measure space, $f: M \to M$  be a measurable transformation,  $\mu$  be a probability measure. If $\varphi: M \to \mathbb{R}$ is a bounded  measurable  function that satisfies the hypothesis $(b)$ or $(c)$. Then

\begin{center}
 $\int \varphi_{-} d\mu = \lim \limits_{n}  \frac{1}{n}\int \sum\limits_{j=0}^{n-1} \varphi \circ f^{j} d\mu  =  \inf \limits_{n}  \frac{1}{n}\int \sum\limits_{j=0}^{n-1} \varphi \circ f^{j} d\mu  $.
\end{center}
\end{maincorollary}

\begin{remark}
In \cite{CM12}, Carvalho and Moreira introduced the notion of half-invariant measure $\mu$, that means that \begin{equation}\label{eq:CM} \mu(f^{-1}(B)) \leq \mu (B)\end{equation} for all measurable set $B$. Note that this implies that $(\mu(f^{-j}(B)))_{j \in \mathbb{N}}$ is a decreasing sequence. The authors showed that for any bounded observable $\varphi:M \to \mathbb{R}$ and a half-invariant measure,  the limit $\lim \limits_{n}  \frac{1}{n} \sum\limits_{j=0}^{n-1} \varphi \circ f^{j} (x) $ exists  for $\mu$ a.e. point $x$ in $M$.  Here,  Corollary \ref{corA}  tell us that  condition \ref{eq:CM} can be relaxed to consider only  the sets of the type $M\setminus E_{k}^{\varepsilon}$ for any $\varepsilon>0$ and $k\in \mathbb{N}$.

\end{remark}

Let $(W,d)$ be a metric space,    $g: W \to W$  be a function, and $x \in W$.  The set $\mathcal{O}^{+}x $ is the  \emph{ forward orbit of $x$}, and it is given by $\mathcal{O}^{+}x :=\{g^{n}(x)\}_{n \in \mathbb{N}}$. A point $x \in W$ is a \emph{ periodic point} if there exists $m \in \mathbb{N}$ such that $g^{m}x =x$. More generally, we say that a point $x \in W$ is \emph{eventually periodic} if there exists $j_{0} \in \mathbb{N}$ such that $g^{j_{0}}x$ is a periodic point.

Let  $S$ be a subset of $W$, and let $g: W\to W$ be  a continuous function. The $\omega$-limit of $S$, denoted by $\omega(S,g)$, is the set of points $y \in W$ for which there are $z \in S$ and a strictly increasing sequence of natural number $\{n_{k}\}_{k \in \mathbb{N}}$ such that $g^{n_{k}} z \to y$ as $k \to \infty$. Note that $\omega(S,g) = \bigcup\limits_{z \in S} \omega(\{z\},g)$.

Let us mention one important consequence of Corollary \ref{corA}.

\begin{corollary}\label{teo:gooda}
Let $(M, \mathcal{A})$  be a measurable space for $M$ metric space,  $f: M \to M$  be a measurable transformation, and $\varphi: M \to \mathbb{R}$ be a bounded  measurable  function.  If one of the following conditions is true

\begin{itemize}
\item[($i$)]  $\lim\limits_{k \to \infty}\limsup\limits_{n }\frac{1}{n}\sum\limits_{i=0}^{n-k-1}  \delta_{x}(f^{-i}(M\setminus E_{k}^{\frac{1}{\ell}})) =0$ for each $\ell \in \mathbb{N} \setminus \{0\}$ where  $\delta_{x}$ the Dirac measure of point $x \in M$;

\item[($ii$)] Suppose that there exists $x \in M$ such that  for any $\varepsilon>0$  there exist $j_{\varepsilon},k_{\varepsilon} \in \mathbb{N}$ satisfying that  $f^{j} (x) \in E^{\varepsilon}_{k_{\varepsilon}} $ for $j \geq j_{\varepsilon}$;

\item[($iii$)] If $M$ is a compact metric space, and there exists $x \in M$ such that for any $\varepsilon >0$ there exists $k_{\varepsilon} \in \mathbb{N}$ satisfying that $\omega(\{x\},f)$ is contained in the interior of $E^{\varepsilon}_{k_{\varepsilon}}$;
\item[($iv$)] Suppose that $M$ is a compact metric space, $f,\varphi, \varphi_{-}$ are continuous functions, and  $\omega(\{x\},f)$ is a finite set for some $x \in M$.

\end{itemize}

Then the limit  $   \lim \limits_{n\to \infty}  \frac{1}{n}  \sum\limits_{j=0}^{n-1} \varphi \circ f^{j}(x) $ exists.

\end{corollary}

We are going to obtain a version of item $(ii)$ of Corollary \ref{teo:gooda} for continuous flow on compact metric spaces.  Let $X:M\times \mathbb{R} \to M$ be a continuous flow, and $M$ to be a compact metric space. Consider $X_{t}:M\to M$ given by $X_{t}(x)= X(t,x)$, and $f_{t}:M \to M$ defined by $f_{t}=X_{t}$. Let $\varphi: M \to \mathbb{R}$ be a bounded  measurable  function and for each $x \in M$ denote the Dirac measure of point $x$ by $\delta_{x}$. We consider the following objects:

$\varphi_{*,-} (y) = \liminf\limits_{n \to \infty} \frac{1}{n} \int_{0}^{n} \varphi \circ f_{t} (y) dt$ for each $y \in M$;

$E^{*,\varepsilon}_{k_{\varepsilon}} = \{ y \in M:  \frac{1}{n} \int_{0}^{n} \varphi \circ f_{t} (y) dt \leq \varphi_{*,-} (y) + \varepsilon$ for some $n \in  \{1,\cdots,k\} \}$.

The  next result is the version of Corollary \ref{teo:gooda} for continuous flow on compact metric spaces.

\begin{corollary}\label{teo:gooda11}
Let $\varphi: M \to \mathbb{R}$ be a bounded  measurable  function, and fix $x \in M$. If  for any $\varepsilon>0$  there exist $t_{\varepsilon} \in \mathbb{R}$  and $k_{\varepsilon} \in \mathbb{N}$ satisfying that  $\delta_{x}(f_{-j}( E^{\varepsilon,*}_{k_{\varepsilon}})) =1$  for $j \geq t_{\varepsilon}$ and $j \in \mathbb{N}$, then the limit  $   \lim \limits_{T\to \infty}  \frac{1}{T}  \int_{0}^{T} \varphi \circ f_{t}(x) dt $  exists.
\end{corollary}

If  $\varphi$  is a continuous function on a compact metric space, we obtain an interesting criterion to provide the existence of Birkhoff's limit as follows.

\begin{maintheorem}\label{flowbirkhoff}
Suppose that $M$ is a compact metric space,  $\varphi:M \to \mathbb{R}$ is a continuous function, and fix $x \in M$. If  $ \limsup\limits_{n} \frac{1}{n} \int_{0}^{n} \varphi \circ f_{t} (y_{x}) dt  \leq    \liminf\limits_{n}  \frac{1}{n} \int_{0}^{n} \varphi \circ f_{t} (x) dt $ for all $y_{x} \in \omega(x)$, then  the limit  $   \lim \limits_{T\to \infty}  \frac{1}{T}  \int_{0}^{T} \varphi \circ f_{t}(x) dt $  exists.
\end{maintheorem}

  We say that  $x \in M$ is a $2d$-point if for any $y_{x} \in \omega (x)$ we have that $\omega(y_{x})$ is a fixed point (i.e.,  there exists $q \in M$ such that $\omega(y_{x}) = \{q\}$ and $f_{t} (q) = q$  for all $t \in \mathbb{R}$). Define the fixed point set under $X$ by $\Fix X = \{q \in M : q$ is a fixed point$\}$.

Let $M$ be a compact metric space $M$, and  $\varphi:M \to \mathbb{R}$ be a continuous function. 

We finish this section presenting the next consequence of Theorem \ref{flowbirkhoff}.

\begin{corollary}\label{cor.app}
Suppose that $M$ is a compact metric space,  $\varphi:M \to \mathbb{R}$ is a continuous function, and take  a $2d$-point $x \in M$. Suppose that $\varphi$ satisfies that $\varphi|_{\omega(x)\cap \Fix X} \equiv \min \varphi$. Then the limit  $   \lim \limits_{T\to \infty}  \frac{1}{T}  \int_{0}^{T} \varphi \circ f_{t}(x) dt $  exists.
\end{corollary}

In the next section, we show some application of our results.

\subsection{Application}

In an example of Bowen, on a compact subset of $\mathbb{R}^{2}$ denoted by $E_{B}$, if $(f_{t}(x))_{t\geq 0}$ converges to a cycle and $\varphi$ is a continuous function on the plane, by taking different values in the saddle points $A$ and $B$, the time average

\begin{center}
$   \lim \limits_{T\to \infty}  \frac{1}{T}  \int_{0}^{T} \varphi \circ f_{t}(x) dt $
\end{center}

\begin{flushleft}
does not exist. 

This means that in this example there is an open set of initial states such that the corresponding orbits define non-stationary time series (whenever one uses an observable which has different values in two saddle points).

\end{flushleft}

\begin{center}
\includegraphics[scale=0.2]{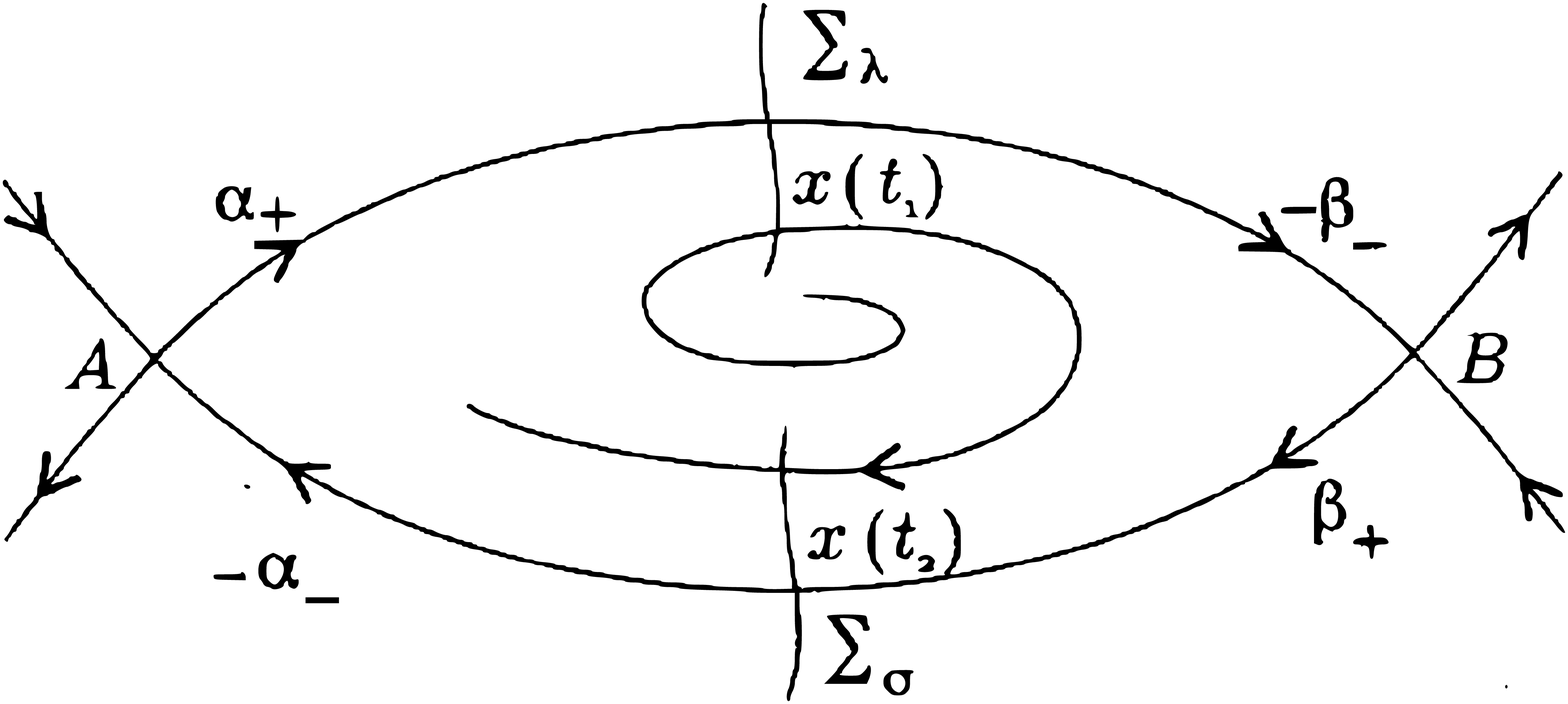}
\end{center}
\begin{center}
Figure 1: Phase portrait of the example by Bowen.
\end{center}

We denote, for the example given in figure 1, the expanding and contracting eigenvalues of the linearized vector field in $A$ by $\alpha_{+}$ and $\alpha_{-}$, and in $B$ by $\beta_{+}$ and $\beta_{-}$. We recall that the saddle points are denoted by $A$ and $B$. The condition which makes the cycle attracting is that the contracting eigenvalues dominate: $\alpha_{-}\beta_{-}> \alpha_{+}\beta_{+}$.

The modulus associated to the upper, respectively lower, saddle connection is denoted by $\lambda$, respectively $\sigma$. They are defined by
\begin{center}
 $\lambda = \alpha_{-}/\beta_{+}$ and $\sigma = \beta_{-}/\alpha_{+}$,
\end{center}
their values are positive and their products is bigger than 1, assuming the cycle to be attracting. 

Applying Corollary \ref{cor.app}, we get the following result.

\begin{corollary}\label{cor.bowen}
Suppose that $\varphi: E_{B} \to \mathbb{R}$ is a continuous function with $\varphi(A) =\varphi(B)= \min \varphi$, and $(f_{t}(x))_{t\geq 0}$ is an orbit converging to the cycle. Then the limit  $\lim \limits_{T\to \infty}  \frac{1}{T}  \int_{0}^{T} \varphi \circ f_{t}(x) dt$ exists.
\end{corollary}

We note that this result provides some information about the existence of Birkhoff's Limit if we take a continuous function.

\subsection{Organization of the text}
In the present section \ref{sec:statement-result1}, we provided preliminary definitions in order to present the statements of the main results together with the application. In Section \ref{sec:proofs}, we give the proofs of our results, divided into three subsections \ref{sec:teo:gooda}, \ref{proff.flows} and \ref{sec:mtheo0a}. In subsection \ref{sec:teo:gooda}, we proved the Corollary \ref{teo:gooda}. In subsection  \ref{proff.flows}, the Corollary \ref{teo:gooda11},  Theorem  \ref{flowbirkhoff} and Corollary \ref{cor.app}  are demonstrated.  Finally, in subsection \ref{sec:mtheo0a}, we showed Theorem \ref{mtheo0a}.

\subsection*{Acknowledgements}
\label{sec:acknowledgements}

This work is part of the PhD thesis of V. Coelho developed at the
Mathematics Department of Federal University of Bahia (UFBA) under
the guidance of L. Salgado. The authors would like to thank the
facilities provided by the Mathematics Institute, the PhD Program
and the financial support from several federal and state
agencies to this Research Program.

\section{Proof of results}\label{sec:proofs}

In this section, we present the proofs of our results.

\subsection{Proof of Corollary \ref{teo:gooda}}\label{sec:teo:gooda}

Let $(M, \mathcal{A})$ be a measurable space for $M$ compact metric space, and $f: M \to M$ be a measurable transformation, $\varphi: M \to \mathbb{R}$ be a bounded measurable function.

We are going to show that if one of the following conditions is true

\begin{itemize}
\item[($i$)] $\lim\limits_{k \to \infty}\limsup\limits_{n }\frac{1}{n}\sum\limits_{i=0}^{n-k-1}  \delta_{x}(f^{-i}(M\setminus E_{k}^{\frac{1}{\ell}})) =0$ for each $\ell \in \mathbb{N} \setminus \{0\}$ where  $\delta_{x}$ the Dirac measure of point $x \in M$;
\item[($ii$)] Suppose that there exists $x \in M$ such that  for any $\varepsilon>0$  there exist $j_{\varepsilon},k_{\varepsilon} \in \mathbb{N}$ satisfying that  $f^{j} x \in E^{\varepsilon}_{k_{\varepsilon}} $ for $j \geq j_{\varepsilon}$;
\item[($iii$)] If $M$ is a compact metric space, and there exists $x \in M$ such that for any $\varepsilon >0$ there exists $k_{\varepsilon} \in \mathbb{N}$ satisfying that $(\mathcal{O}^{+} x)'$ is contained in the interior of $E^{\varepsilon}_{k_{\varepsilon}}$;
\item[($iv$)] Suppose that $M$ is a compact metric space, $f,\varphi, \varphi_{-}$ are continuous functions, and  $(\mathcal{O}^{+} x)'$ is a finite set for some $x \in M$.
\end{itemize}

Then the limit $\lim \limits_{n\to \infty}  \frac{1}{n}  \sum\limits_{j=0}^{n-1} \varphi \circ f^{j}(x)$ exists.

\begin{proof}[Proof of Corollary \ref{teo:gooda}]

Fix $\varphi : M \to \mathbb{R}$, and consider $(\varphi_{n})_{n}$ to be the additive sequence for $f$  given by $\varphi_{n} := \sum\limits_{j=0}^{n-1} \varphi \circ f^{j}$ for each $n$ in $\mathbb{N}$. Consider  $\varphi_{-}: M \to \mathbb{R}$  given by $\varphi_{-}(w) = \liminf\limits_{n} \frac{1}{n} \sum\limits_{j=0}^{n-1} \varphi \circ f^{j} (w)$ for $w \in M$.

For each $\varepsilon >0$ fixed and $k \in \mathbb{N}$ we define

\begin{center}
$E_{k}^{\varepsilon}=\{w \in M:\varphi_{j}(w) \leq j (\varphi_{-}(w)+\varepsilon)$ for some $j \in \{1,...,k\}\}$.
\end{center}

Consider the measure $\mu = \delta_{x}$ where $\delta_{x}$ is the Dirac measure of point $x$.

To apply Corollary \ref{corA} it is sufficient to prove that $\varphi$ satisfies condition $(b)$, i.e.,
\begin{align}
\lim\limits_{k \to \infty}\limsup\limits_{n }\frac{1}{n}\sum\limits_{i=0}^{n-k-1}  \delta_{x}(f^{-i}(M\setminus E_{k}^{\frac{1}{\ell}})) =0, \forall \ell \in \mathbb{N} \setminus \{0\}. 
\end{align}

$(ii)$ If $x$ is an eventually periodic point, there is nothing to show. Suppose that $x$ is not an eventually periodic point. Fixed $\varepsilon>0$,  there exist $j_{\varepsilon},k_{\varepsilon} \in \mathbb{N}$ such that $f^{j} x \in E^{\varepsilon}_{k_{\varepsilon}} $ for $j \geq j_{\varepsilon}$. This implies that $(\mathcal{O}^{+} x) \cap M \setminus E_{k}^{\varepsilon}$ is a finite set for $k \geq k_{\varepsilon}$.

\textbf{Claim 1:} $\{j \in \mathbb{N}:  x \in f^{-j}(M\setminus E_{k}^{\varepsilon})\}$ is a finite set.

Suppose the claim would be false. Then we could find  a sequence $(j_{s})_{s \in \mathbb{N}}$ such that $f^{j_{s}}x \in M \setminus E_{k}^{\varepsilon}$  for all $s \in \mathbb{N}$. So $f^{j_{s}}(x) \in  (\mathcal{O}^{+}x) \cap M \setminus E_{k}^{\varepsilon}$ for all $s  \in \mathbb{N}$. Then for $s> \# ((\mathcal{O}^{+}x) \cap M \setminus E_{k}^{\varepsilon})$ there exists $t \in \mathbb{N}$ such that $t < s$  and  $f^{j_{s}}x = f^{j_{t}}x$. Using that $x$ is not an eventually periodic point, we are done.

By Claim 1, there exists $j_{0} \in \mathbb{N}$ such that for $j \geq j_{0}$ we must have the following  $x \in  M \setminus f^{-j}(M\setminus E_{k}^{\varepsilon})$, and then $\mu  (f^{-j}(M\setminus E_{k}^{\varepsilon})) = 0$ for $j \geq j_{0}$.

Using that $E_{k}^{\varepsilon} \subseteq E_{k+1}^{\varepsilon}$, we see that $\mu  (f^{-j}(M\setminus E_{\widetilde{k}}^{\varepsilon})) = 0$ for all $\widetilde{k}  \geq k_{\varepsilon}$ and $j \geq j_{0}$.

Now, take $\widetilde{k}$  such that $\widetilde{k}+1 > j_{0}$ . It easy to see that there exists $n \in \mathbb{N}$ such that  $n > j_{0}+\widetilde{k}+1$, and note that

\begin{center}
$ \frac{1}{n}\sum\limits_{j=0}^{n-\widetilde{k}-1}  \mu(f^{-j}(M\setminus E_{\widetilde{k}}^{\varepsilon})) = \frac{1}{n}\sum\limits_{j=0}^{j_{0}-1}  \mu(f^{-j}(M\setminus E_{\widetilde{k}}^{\varepsilon})) + \frac{1}{n}\sum\limits_{j=j_{0}}^{n-\widetilde{k}-1}  \mu(f^{-j}(M\setminus E_{\widetilde{k}}^{\varepsilon}))  $
\end{center}

Using that $\mu  (f^{-j}(M\setminus E_{\widetilde{k}_{\varepsilon}}^{\varepsilon})) = 0$ for all $\widetilde{k}  \geq k_{\varepsilon}$ and $j \geq j_{0}$,
\begin{center}

$0 \leq \frac{1}{n}\sum\limits_{j=0}^{n-\widetilde{k}-1}  \mu(f^{-j}(M\setminus E_{\widetilde{k}}^{\varepsilon})) = \frac{1}{n}\sum\limits_{j=0}^{j_{0}-1}  \mu(f^{-j}(M\setminus E_{\widetilde{k}}^{\varepsilon})) \leq \frac{j_{0}}{n}$,
\end{center}

and then

\begin{center}
$\limsup\limits_{n }\frac{1}{n}\sum\limits_{i=0}^{n-\widetilde{k}-1}  \mu(f^{-i}(M\setminus E_{\widetilde{k}}^{\varepsilon})) = 0$ for $\widetilde{k}+1 > j_{0}$,
\end{center}

This implies that \begin{equation}
\lim\limits_{k \to \infty}  \limsup\limits_{n }\frac{1}{n}\sum\limits_{i=0}^{n-k-1}  \mu(f^{-i}(M\setminus E_{k}^{\varepsilon})) = 0,
\end{equation}

this completes the proof of item $(ii)$.

$(iii)$ Suppose that $M$ is a compact metric space.  For any $\varepsilon >0$ there exists $k_{\varepsilon} \in \mathbb{N}$ such that  $\omega(\{x\},f)$ is contained in the interior of $E^{\varepsilon}_{k_{\varepsilon}}$.

 We are going to verify condition ($ii$).

\textbf{Claim 2:} For $k \geq k_{\varepsilon}$,  $(\mathcal{O}^{+} x) \cap M \setminus E_{k}^{\varepsilon}$ is a finite set.

Suppose, on the contrary to our claim, that there exists a sequence $\{n_{s}\}_{s \in \mathbb{N}}$ such that $f^{n_{s}}x \notin E_{k}^{\varepsilon}$. By compactness of $M$, there exists a subsequence of sequence $(f^{n_{\ell}} x)_{\ell \in \mathbb{N}}$ that converges to some $p \in M$. Without loss generality, the sequence  converges to  $p\in \omega (\{x\},f) $, so $p$ is an element of interior of $E^{\varepsilon}_{k_{\varepsilon}}$,i.e., $p \in \intt  E^{\varepsilon}_{k_{\varepsilon}}$. Using that $\intt  E^{\varepsilon}_{k_{\varepsilon}}$ is an open set, there exists $n_{p}>0$ such that for $n_{s} \geq n_{p}$ we have that $f^{n_{s}}x \in \intt  E^{\varepsilon}_{k_{\varepsilon}}$. But $\intt  E^{\varepsilon}_{k_{\varepsilon}} \subseteq  E^{\varepsilon}_{k_{\varepsilon}} \subseteq  E^{\varepsilon}_{k}$ and $f^{n_{s}}x \notin E_{k}^{\varepsilon}$ for all  $s \in \mathbb{N}$, this contradiction concludes the proof of the Claim 2, and we are done.

($iv$) For each $\varepsilon >0$ fixed and $k \in \mathbb{N}$ we define

\begin{center}
$\widehat{E_{k}^{\varepsilon}}=\{w \in M:\varphi_{j}(w) < j (\varphi_{-}(w)+\varepsilon)$ for some  $j \in \{1,...,k\}\}$
\end{center}

where $\widehat{E_{k}^{\varepsilon}} \subseteq E_{k}^{\varepsilon}$.

By continuity of $f$,$\varphi$ and $\varphi_{-}$, we see that $\widehat{E_{k}^{\varepsilon}}$ is an open set of $M$. Using that $M= \bigcup\limits_{k\in \mathbb{N}}\widehat{E_{k}^{\varepsilon}}$ and  $\omega(\{x\},f)$ is a finite set, there exists $k_{\varepsilon}$ such that  $\omega(\{x\},f) \subseteq \widehat{E_{k_{\varepsilon}}^{\varepsilon}}$. By item $(iii)$, we are done. This completes the proof of Corollary \ref{teo:gooda}.

\end{proof}

\subsection{Continuous flow on compact metric spaces}\label{proff.flows}

\begin{proof}[Proof of Corollary \ref{teo:gooda11}]
Let $\varphi: M \to \mathbb{R}$ be a bounded function, define $\psi: M \to \mathbb{R}$ by $\psi(y) = \int_{0}^{1}\varphi \circ f_{t}(y) dt$ for each $y \in M$. Fix $T >0$, and note that

\begin{center}
$\frac{1}{T}  \int_{0}^{T} \varphi \circ f_{t}(y) dt = \frac{1}{T}  \sum\limits_{j=0}^{[T]-1} \int_{j}^{j+1} \varphi \circ f_{t}(y) dt + \frac{1}{T}   \int_{[T]}^{T} \varphi \circ f_{t}(y) dt  $
\end{center}

Considering $t= j+s$ for $s \in [0,1]$, we see that \begin{eqnarray}
 \int_{j}^{j+1} \varphi \circ f_{t}(y) dt & = & \int_{0}^{1} \varphi \circ f_{s}(f_{j}y) ds  \nonumber \\
\frac{1}{T}  \int_{0}^{T} \varphi \circ f_{t}(y) dt  & = & \frac{1}{T}  \sum\limits_{j=0}^{[T]-1} \int_{0}^{1} \varphi \circ f_{t}(f_{j}y) dt + \frac{1}{T}   \int_{[T]}^{T} \varphi \circ f_{t}(y) dt  \nonumber \\
\frac{1}{T}  \int_{0}^{T} \varphi \circ f_{t}(y) dt & = &  \frac{1}{T}  \sum\limits_{j=0}^{[T]-1} \psi \circ f_{j}(y)  + \frac{1}{T}   \int_{[T]}^{T} \varphi \circ f_{t}(y) dt \label{eq.41}
\end{eqnarray}

Take $T = n$, by equation (\ref{eq.41}), $\frac{1}{n} \sum\limits_{j=0}^{n-1} \psi \circ f_{j} (y)  =\frac{1}{n}  \int_{0}^{n} \varphi \circ f_{t}(y) dt  $.

Observe that $\psi$ is a bounded function since $\varphi$ is a bounded function. Recall that $\psi_{-} (y) = \liminf\limits_{n} \frac{1}{n} \sum\limits_{j=0}^{n-1} \psi \circ f_{j} (y)$ for $y \in M$. Now,  for each $\varepsilon>0$ and $k \in \mathbb{N}$ define

$\widetilde{E}^{\varepsilon}_{k} = \{y \in M: \frac{1}{n} \sum\limits_{j=0}^{n-1} \psi \circ f_{j} (y) \leq  \psi_{-} (y) + \varepsilon$ for some $n \in  \{1,\cdots, k\} \} $.

$\widetilde{E}^{\varepsilon}_{k} = \{y \in M: \frac{1}{n}  \int_{0}^{n} \varphi \circ f_{t}(y) dt  \leq   \varepsilon+  \liminf\limits_{n} \frac{1}{n}  \int_{0}^{n} \varphi \circ f_{t}(y) dt $ for some $n \in  \{1,\cdots, k\} \} $, so  $\widetilde{E}^{\varepsilon}_{k} = E^{\varepsilon,*}_{k_{\varepsilon}}$.

By hypothesis, there exist $t_{\varepsilon} \in \mathbb{R}$  and $k_{\varepsilon} \in \mathbb{N}$ satisfying that  $f_{j} (x) \in E^{\varepsilon,*}_{k_{\varepsilon}} $ for $j \geq t_{\varepsilon}$ and $j \in \mathbb{N}$. Then, by Corollary \ref{teo:gooda}, the limit  $\lim\limits_{T \to \infty}\frac{1}{T}  \sum\limits_{j=0}^{[T]-1} \psi \circ f_{j}(y)= \lim\limits_{T \to \infty}\frac{1}{[T]}  \sum\limits_{j=0}^{[T]-1} \psi \circ f_{j}(y)$ exists since \begin{center}
$\frac{1}{T}  \sum\limits_{j=0}^{[T]-1} \psi \circ f_{j}(y) = \frac{1}{[T] + \beta_{T}}  \sum\limits_{j=0}^{[T]-1} \psi \circ f_{j}(y) = \frac{1}{[T](1  + \frac{\beta_{T}}{[T]\label{bivectparthyp}})}  \sum\limits_{j=0}^{[T]-1} \psi \circ f_{j}(y) $
\end{center} for some $\beta_{T} \in (0,1]$ such that $T= [T] + \beta_{T}$.

Note that

\begin{align*}
\big|\frac{1}{T}   \int_{[T]}^{T} \varphi \circ f_{t}(y) dt\big| = \big|\frac{1}{T}   \int_{0}^{T-[T]} \varphi \circ f_{t} (  f_{[T]}(y)) dt\big|    \leq   \frac{1}{T}   \int_{0}^{T-[T]} |\varphi \circ f_{t} (  f_{[T]}(y))| dt   \leq \\
 \frac{1}{T}   \int_{0}^{1} |\varphi \circ f_{t} (  f_{[T]}(y))| dt \leq   \frac{\|\varphi\|}{T}  \to 0,
\end{align*}

as $T$ tends to infinity.
\end{proof}

If $\varphi$ is a continuous function, we obtain an interesting criterion to provide the existence of Birkhoff's limit as follows.

\begin{proof}[Proof of Theorem \ref{flowbirkhoff}]

Suppose that $M$ is a compact metric space, and $\varphi:M \to \mathbb{R}$ is a continuous function (so $\varphi$ is a bounded function). For some $x \in M$ we have that $ \limsup\limits_{n} \frac{1}{n} \int_{0}^{n} \varphi \circ f_{t} (y_{x}) dt  \leq    \liminf\limits_{n}  \frac{1}{n} \int_{0}^{n} \varphi \circ f_{t} (x) dt $ for all $y_{x} \in \omega(x)$. In view of Corollary \ref{teo:gooda11}, it is sufficient to show that  for any $\varepsilon>0$  there exist $t_{\varepsilon} \in \mathbb{R}$  and $k_{\varepsilon} \in \mathbb{N}$ satisfying that  $\delta_{x}(f_{-j}( E^{\varepsilon,*}_{k_{\varepsilon}})) =1$  for $j \geq t_{\varepsilon}$ and $j \in \mathbb{N}$.

Suppose, by contradiction, that  there exists $\varepsilon >0$ such that for any $k \in \mathbb{N}$, and for any $t \in \mathbb{R}$ there exists $j_{k} \in \mathbb{N}$ with $j_{k} > t$ such that $f_{j_{k}} (x) \notin E^{*,\varepsilon}_{k}$.

 In particular, for each $k \in \mathbb{N}$, taking $t=k$,  there exists  $j_{k}>k$ and $j_{k} \in \mathbb{N}$ such that $f_{j_{k}} (x) \notin E^{*,\varepsilon}_{k}$. This implies that $j_{k} \to +\infty$ as $k$  tends to infinity.

By compactness of $M$, there exists a subsequence of $(f_{j_{k}} (x))_{k \in \mathbb{N}}$ that converges to some $y_{x} \in \omega(x)$, suppose that  \begin{equation}f_{j_{k_{s}}} (x) \to_{s} y_{x}
\end{equation}  where $j_{k_{s}}$ tends to infinity as $s$ tends to infinity,  $f_{j_{k_{s}}} (x) \notin E^{*,\varepsilon}_{k_{s}}$ and $j_{k_{s}}>k_{s}$. Without loss of generality, we may assume that $k_{1} < k_{2} < \cdots < k_{s} < k_{s+1} \cdots$

We recall that $E^{*,\varepsilon}_{k_{\varepsilon}} = \{ y \in M:  \frac{1}{n} \int_{0}^{n} \varphi \circ f_{t} (y) dt \leq \varphi_{*,-} (y) + \varepsilon$ for some $n \in  \{1,\cdots,k_{\varepsilon}\} \}$.

For each $s \in \mathbb{N}$, by definition of $E^{*,\varepsilon}_{k_{s}}$, \begin{center}
$\frac{1}{n} \int_{0}^{n} \varphi \circ f_{t} (f_{j_{k_{s}}} (x)) dt > \varphi_{*,-} (f_{t_{k_{s}}} (x)) + \varepsilon$
\end{center} for any $n \in  \{1,\cdots,k_{s} \}$ since $f_{j_{k_{s}}} (x) \notin E^{*,\varepsilon}_{k_{s}}$.

Recall that $\varphi_{*,-} (z)  =  \liminf\limits_{n \to \infty} \frac{1}{n} \int_{0}^{n} \varphi \circ f_{t} (z) dt  = \liminf\limits_{n \to \infty} \frac{1}{n} \sum\limits_{j=0}^{n-1} \psi \circ f_{j} (z)= \psi_{-}(z)$, where $\psi$ is a bounded function. Then $\psi_{-}(f_{j} (z))= \psi_{-}(z)$ for all $j \geq 0$ and $z \in M$, so $\varphi_{*,-} (f_{j}z) = \varphi_{*,-} (z)$ for all $j \geq 0$.

 Using that  $k_{1} < k_{2} < \cdots < k_{s} < k_{s+1} \cdots$, we see that $k_{1} \in \{1,\cdots,k_{s} \}$ for any $s \geq 1$, and then
\begin{center}
$\frac{1}{k_{1}} \int_{0}^{k_{1}} \varphi \circ f_{t} (f^{t_{k_{s}}} (x)) dt > \varphi_{*,-} (f^{t_{k_{s}}} (x)) + \varepsilon = \varphi_{*,-} (x) + \varepsilon $.
\end{center}

Recall that  $\frac{1}{n} \sum\limits_{j=0}^{n-1} \psi \circ f_{j} (y)  =\frac{1}{n}  \int_{0}^{n} \varphi \circ f_{t}(y) dt  $ where   $\psi(y) = \int_{0}^{1}\varphi \circ f_{t}(y) dt$ for each $y \in M$. This implies that $\frac{1}{k_{1}} \int_{0}^{k_{1}} \varphi \circ f_{t} (y) dt= \frac{1}{k_{1}} \sum\limits_{j=0}^{k_{1}-1} \psi \circ f_{j} (y) $ for each $y \in M$.

A straightforward calculation shows that $\psi:M \to \mathbb{R}$ is uniformly continuous.

So $\frac{1}{k_{1}} \int_{0}^{k_{1}} \varphi \circ f_{t} (\cdot ) dt= \frac{1}{k_{1}} \sum\limits_{j=0}^{k_{1}-1} \psi \circ f_{j} (\cdot) $  is a continuous function. By continuity of $\frac{1}{k_{1}} \int_{0}^{k_{1}} \varphi \circ f_{t} (\cdot ) dt$, we have that

\begin{center}
$\frac{1}{k_{1}} \int_{0}^{k_{1}} \varphi \circ f_{t} (f^{t_{k_{s}}} (x)) dt \to \frac{1}{k_{1}} \int_{0}^{k_{1}} \varphi \circ f_{t} (y_{x}) dt   \geq  \varphi_{*,-} (x) + \varepsilon $.
\end{center}

So using that  $k_{1} < k_{2} < \cdots < k_{s} < k_{s+1} \cdots$, we see that $k_{\ell} \in \{1,\cdots,k_{s} \}$ for any $s \geq \ell$ for each $\ell \in \mathbb{N}$, and then

\begin{center}
$\frac{1}{k_{\ell}} \int_{0}^{k_{\ell}} \varphi \circ f_{t} (f^{t_{k_{s}}} (x)) dt \to \frac{1}{k_{\ell}} \int_{0}^{k_{\ell}} \varphi \circ f_{t} (y_{x}) dt   \geq  \varphi_{*,-} (x) + \varepsilon $.
\end{center}

This implies that $\frac{1}{k_{\ell}} \int_{0}^{k_{\ell}} \varphi \circ f_{t} (y_{x}) dt \geq \varphi_{*,-} (x) + \varepsilon $ for any $\ell \in \mathbb{N}$, and then

\begin{center}
 $\limsup\limits_{n} \frac{1}{n} \int_{0}^{n} \varphi \circ f_{t} (y_{x}) dt  \geq   \varphi_{*,-}(x) + \varepsilon  >\varphi_{*,-}(x) =    \liminf\limits_{n}  \frac{1}{n} \int_{0}^{n} \varphi \circ f_{t} (x) dt $,
\end{center} and we are done.

\end{proof}

Here, we recall the following lemma.

\begin{lemma}\label{lem.fixed.point}
Suppose that $M$ is a compact metric space, and  $\varphi:M \to \mathbb{R}$ is a continuous function. If $p,q \in M$ and $\omega(p)=\{q\}$ then  $   \lim \limits_{T\to \infty}  \frac{1}{T}  \int_{0}^{T} \varphi \circ f_{t}(p) dt =\varphi(q)$.
\end{lemma}

\begin{proof}[Proof of Corollary \ref{cor.app}]
For each $y_{x} \in \omega(x)$, there exists a fixed point $q_{y_{x}}$ such that $\omega(y_{x}) = \{q_{y_{x}}\}$, so $q_{y_{x}} \in \omega(x)\cap \Fix X$, and then $\varphi(q_{y_{x}})  = \min\varphi$.

Now, by Lemma \ref{lem.fixed.point}, $   \lim \limits_{T\to \infty}  \frac{1}{T}  \int_{0}^{T} \varphi \circ f_{t}(y_{x}) dt =\varphi(q_{y_{x}}) = \min \varphi$.

Note that $\min \varphi \leq \liminf\limits_{n}  \frac{1}{n} \int_{0}^{n} \varphi \circ f_{t} (x) dt$, and then, by Theorem \ref{flowbirkhoff},  we are done.

\end{proof}

\subsection{Proof of Theorem \ref{mtheo0a}}\label{sec:mtheo0a}

Let $(M, \mathcal{A}, \mu )$  be a measure space, $f: M \to M$   be a measurable function,  $\mu$ be a finite measure. Suppose that $(\varphi_{n})_{n}$ is a subadditive sequence for $f$ such that $\varphi_{1} \leq \beta$ for some $\beta \in \mathbb{R}$. Without loss of generality, we assume that $\beta >0$.

 Under the conditions stated above, and supposing that the following conditions are satisfied:
\begin{itemize}

\item[($a$)] for all $j \in \mathbb{N}$ we have that $\varphi_{-}(f^{j}(x)) = \varphi_{-}(x)$  $\mu-$almost everywhere $x$ in $M$;

\item[($b$)]  $\lim\limits_{k \to \infty}\limsup\limits_{n }\frac{1}{n}\sum\limits_{i=0}^{n-k-1}  \mu(f^{-i}(M\setminus E_{k}^{\frac{1}{\ell}})) =0$ for each $\ell \in \mathbb{N} \setminus \{0\}$.
\end{itemize}

Then Theorem \ref{mtheo0a} ensures that $\inf\limits_{n}   \frac{1}{n}\int \varphi_{n} d\mu = \int \varphi_{-} d\mu$. Moreover, if there exists $\gamma >0$ such that for all $n >0$, $\frac{\varphi_{n}}{n} \geq -\gamma $ then
\begin{center}
$\int \varphi_{-} d\mu  = \inf\limits_{n}   \frac{1}{n}\int \varphi_{n} d\mu = \lim \limits_{n}   \frac{1}{n}\int \varphi_{n} d\mu $.
\end{center}

The proof will be divided into two steps. In first step, we show the particular version of Theorem \ref{mtheo0a} when the  sequence  $(\frac{\varphi_{n}}{n})_{n} $ is uniformly bounded from below, i.e., there exists $\alpha >0$ such that $\frac{\varphi_{n}}{n} \geq -\alpha $ for all $n \in \mathbb{N}$. In the second step, using a truncation argument, we conclude from step 1 the proof of the Theorem.

We begin by proving the following theorem.

\begin{theorem}\label{mtheo11a}
Let $(M, \mathcal{A}, \mu )$  be a measure space, $f: M \to M$   be a measurable function,  $\mu$ be a finite measure. Suppose that  $(\varphi_{n})_{n}$ is a subadditive sequence for $f$ such that $\varphi_{1} \leq \beta$ for some $\beta >0$. If the following conditions are satisfied:
\begin{itemize}

\item[($a$)] for all $j \in \mathbb{N}$ we have that $\varphi_{-}(f^{j}(x)) = \varphi_{-}(x)$  $\mu-$almost everywhere $x$ in $M$;

\item[($b$)]  $\lim\limits_{k \to \infty}\limsup\limits_{n }\frac{1}{n}\sum\limits_{i=0}^{n-k-1}  \mu(f^{-i}(M\setminus E_{k}^{\frac{1}{\ell}})) =0$ for each $\ell \in \mathbb{N} \setminus \{0\}$;

\item[($d$)] there exists $\gamma >0$ such that for all $n >0$, $\frac{\varphi_{n}}{n} \geq -\gamma $.

\end{itemize}

Then  $\lim \limits_{n} \frac{1}{n}\int \varphi_{n} d\mu =\inf\limits_{n}   \frac{1}{n}\int \varphi_{n} d\mu = \int \varphi_{-} d\mu$.

\end{theorem}

\begin{proof}[Proof of Theorem \ref{mtheo11a}] First, without loss of generality, we consider $\beta = \gamma$. Using that $(\varphi_{n})_{n}$ is a subadditive sequence for $f$, we obtain that $\varphi_{m} \leq \sum\limits_{j=0}^{m-1} \varphi_{1} \circ f^{j}$ for all $m \in \mathbb{N}$,  but $\varphi_{1} \leq \beta$,   so $-\beta \leq\frac{\varphi_{m}}{m} \leq \beta$, and $ -\beta \leq \frac{1}{m} \int \varphi_{m} d\mu \leq  \beta$ for all $m$ in $\mathbb{N}$. In particular, $ \varphi_{1}$ is integrable. Define $\varphi_{-}: M \to [-\beta, \beta]$ by $\varphi_{-}(x) = \liminf\limits_{n} \frac{\varphi_{n}(x)}{n}$. So  $\beta \geq \varphi_{-}(x) \geq -\beta$ for all $x$ in $M$, and then $\varphi_{-}$ is integrable.

Fixed $\varepsilon >0$, define for each $k \in \mathbb{N}$

\begin{center}
$E_{k}^{\varepsilon}:=\{x \in M:\varphi_{j}(x) \leq j (\varphi_{-}(x)+\varepsilon)$ for some  $j \in \{1,...,k\}\}$
\end{center}

It is clear that $E_{k}^{\varepsilon} \subseteq E_{k+1}^{\varepsilon}$ for all $k$. Note that by  definition of  $\varphi_{-}$, we have that $M = \bigcup\limits_{k} E_{k}^{\varepsilon}$. Define $\psi_{k}(x) = \varphi_{-}(x) +\varepsilon$ if $x \in E_{k}^{\varepsilon}$, and  $\psi_{k}(x) = \varphi_{1}(x)$ if $x \notin E_{k}^{\varepsilon}$.  Suppose that $x \notin E_{k}^{\varepsilon}$, then $\psi_{k}(x) = \varphi_{1}(x)$, but by $E_{k}^{\varepsilon}$'s definition we have that $\varphi_{1}(x) > \varphi_{-}(x) + \varepsilon$. It imples that $\psi_{k} \geqslant \varphi_{-} + \varepsilon$ in $M$. Now, using that  $M = \bigcup\limits_{k} E_{k}^{\varepsilon}$, we see that $\lim\limits_{k\to \infty} \psi_{k} (x) = \varphi_{-}(x) + \varepsilon$ for each $x \in M$.

Now, let $L$ be a fixed and arbitrary point of accumulation of sequence $(\frac{1}{n} \int \varphi_{n}d\mu)_{n}$, so there exists $(n_{t})_{t \in \mathbb{N}}$ such that $ \lim\limits_{t \to \infty}\frac{1}{n_{t}} \int \varphi_{n_{t}} d\mu  = L$ and $L\in [-\beta,\beta] $. The basic idea of the proof is to verify that  $\int \varphi_{-} d\mu\leq L \leq \lim \limits_{k \to \infty}   \int \psi_{k} d\mu  $. Later,  an easy computation will show that $\int \varphi_{-} d\mu= L$. Observing that $L$ is an arbitrary point of accumulation of sequence $(\frac{1}{n} \int \varphi_{n}d\mu)_{n}$, we conclude that this sequence converges to $\int \varphi_{-} d\mu$. This will end the proof of Theorem \ref{mtheo11a}.

From the above we are going to show that  $\int \varphi_{-} d\mu\leq L$ and  $L\leq \lim \limits_{k \to \infty}   \int \psi_{k} d\mu  $. First, we observe that $\int \varphi_{-} d\mu\leq L$.  By hypothesis, there exists $\beta > 0$ such that $\frac{\varphi_{n}}{n} \geq -\beta$ for all $n$. We have that $\frac{\varphi_{n}}{n} \geq -\beta$. Define $f_{n} (x) := \frac{\varphi_{n}}{n}(x) +\beta \geq 0$ and note that
\begin{align*}
 f(x) = \liminf\limits_{n} (\frac{\varphi_{n}}{n}(x) +\beta) = \varphi_{-}(x) +\beta.
 \end{align*} By Fatou's Lemma, we have that  $f(x) = \varphi_{-}(x) +\beta$ is an  integrable funcion, and
\begin{align*}
\int \liminf\limits_{n} (f_{n}) d\mu \leq \liminf\limits_{n} \int f_{n} d\mu \leq \liminf\limits_{n_{t}} \int f_{n_{t}} d\mu \\
\int \varphi_{-}(x) +\beta d\mu \leq \liminf\limits_{n_{t}} \int (\frac{\varphi_{n_{t}}}{n_{t}} +\beta)d\mu
  \end{align*}
Then
\begin{align*}
 \int \varphi_{-}(x) d\mu \leq \liminf\limits_{n_{t}} \int \frac{\varphi_{n_{t}}}{n_{t}} d\mu = \lim\limits_{n_{t}} \int \frac{\varphi_{n_{t}}}{n_{t}}d\mu = L.
  \end{align*}
So
\begin{equation}\label{eq:int11a}
\int \varphi_{-}(x) d\mu\leq L.
\end{equation}

Now, we show that $L\leq \lim \limits_{k \to \infty}   \int \psi_{k} d\mu  $. We need of the following result.

\begin{lemma}\label{lema1a}
For each $n> k \geq 1$ and $\mu$-a.e. $x \in M$,

\begin{center}
$\varphi_{n}(x) \leq \sum\limits_{i=0}^{n-k-1} \psi_{k}(f^{i}(x)) + \sum\limits_{i=n-k}^{n-1} \max\{\psi_{k}, \varphi_{1}\}(f^{i}(x))$
\end{center}
\end{lemma}
\begin{proof} Use the subadditivity of sequence $(\varphi_{n})_{n}$, and the fact that  $\varphi_{-}$ is invariant in orbit of $x$ in $\mu$-a.e., see Lemma $1$ in \cite{AB09}.
\end{proof}

Note  that $\psi_{k}$  is integrable. We have that $-\beta \leq \frac{\varphi_{n}}{n}$ for all $n$, so $-\beta \leq \varphi_{-}$ and $-\beta \leq \varphi_{1}$. Now, $-\beta < -\beta + \varepsilon \leq \varphi_{-} +\varepsilon$, then $-\beta \leq \psi_{k}$.

Note that $-\beta \leq \psi_{k} \leq \max\{\varphi_{-} + \varepsilon,  \varphi_{1}\} \leq \max\{\varphi_{-} + \varepsilon,  \beta\} $, where $\max\{\varphi_{-} + \varepsilon,  \beta\}$ is integrable, so $\psi_{k}$ is integrable. Note that

$\max\{\varphi_{1}, \psi_{k}\}\circ f^{i} \leq \max\{\varphi_{1}, \varphi_{-} + \varepsilon,  \beta\}\circ f^{i} =  \max\{\varphi_{-} + \varepsilon,  \beta\}\circ f^{i}  = \max\{\varphi_{-} + \varepsilon,  \beta\} $ because $\varphi_{-}$ is invariant in orbit of $x$ in $\mu$-a.e.

But  $\max\{\varphi_{-} + \varepsilon,  \beta\}$ is integrable, then $\max\{\varphi_{1}, \psi_{k}\}\circ f^{i}$ is integrable too for all $i$ in $\mathbb{N}$. By Lemma \ref{lema1a},

\begin{equation}\label{ex1a}
\frac{1}{n}\int\varphi_{n}(x) d\mu \leq  \frac{1}{n}\sum\limits_{i=0}^{n-k-1} \int \psi_{k}(f^{i}(x)) d\mu + \frac{1}{n}\sum\limits_{i=n-k}^{n-1} \int \max\{\psi_{k}, \varphi_{1}\}(f^{i}(x)) d\mu.
\end{equation}

Define $\varphi^{+} = \max\{0,\varphi\}$, and note that  \begin{center}
$\sum \limits_{i=n-k}^{n-1} \int \max\{\psi_{k}, \varphi_{1}\}(f^{i}(x)) d\mu \leq \sum \limits_{i=n-k}^{n-1}  \int \max\{\varphi_{-}+\varepsilon, \varphi_{1}^{+}\}(f^{i}(x)) d\mu $.
\end{center}

Define $S = \{x \in M: \varphi_{-}(x)+\varepsilon \geq \varphi_{1}^{+}(x)  \}$, so

\begin{center}
  $\sum \limits_{i=n-k}^{n-1}  \int \max\{\varphi_{-}+\varepsilon, \varphi_{1}^{+}\}(f^{i}(x)) d\mu  = \sum \limits_{i=n-k}^{n-1} [ \int_{S} \varphi_{-}+\varepsilon d\mu +   \int_{M\setminus S} \varphi_{1}^{+} \circ f^{i} d\mu ]$.
 \end{center}

Using that $-\beta \leq \varphi_{-}$ and  $\int \varphi_{-}d\mu \leq L \in [-\beta, \infty)$, then $\int_{S} \varphi_{-}d\mu < \infty$. So,

  \begin{center}
$ \sum \limits_{i=n-k}^{n-1} [ \int_{S} \varphi_{-}+\varepsilon d\mu +   \int_{M\setminus S} \varphi_{1}^{+} \circ f^{i} d\mu ]  \leq  k [\int_{S} \varphi_{-}+\varepsilon d\mu +  \beta]$,
\end{center} and

\begin{equation}\label{ex2a} \frac{1}{n}\sum \limits_{i=n-k}^{n-1} \int \max\{\psi_{k}, \varphi_{1}\}(f^{i}(x)) d\mu \leq  \frac{k}{n} (\int_{S} \varphi_{-}+\varepsilon d\mu + \beta).
\end{equation}

Now, we are going to show that  \begin{center}
$\frac{1}{n}\sum\limits_{i=0}^{n-k-1} \int \psi_{k}(f^{i}(x)) d\mu \leq  (1-\frac{k}{n}) \int \psi_{k}d\mu  + 2 \beta\cdot\frac{1}{n}\sum\limits_{i=0}^{n-k-1}  \mu(f^{-i}(M \setminus E_{k}^{\varepsilon}))$
\end{center}

Define $F_{i,k} := f^{-i}(E_{k}^{\varepsilon})$ for each $i \in \{0,...,n-k-1\}$, so

$\int \psi_{k}(f^{i}(x)) d\mu = \int_{F_{i,k}} \varphi_{-}(f^{i}(x)) + \varepsilon d\mu + \int_{M \setminus F_{i,k}} \psi_{k}(f^{i}(x)) d\mu $. Using that $\varphi_{-}$ is invariant in orbit of $x$ in $\mu$-a.e., we have that \begin{center}
$\int \psi_{k}(f^{i}(x)) d\mu = \int_{F_{i,k}} \varphi_{-}(x) + \varepsilon d\mu + \int_{M \setminus F_{i,k}} \psi_{k}(f^{i}(x)) d\mu$.
\end{center}

But $\varphi_{-}(x) + \varepsilon \leq \psi_{k}$ in $M$, so
\begin{align*}
\int \psi_{k} \circ f^{i} d\mu \leq \int_{F_{i,k}} \psi_{k} d\mu + \int_{M \setminus F_{i,k}} \psi_{k}\circ f^{i} d\mu =\\
 \int_{F_{i,k}} \psi_{k} d\mu +  \int_{M \setminus F_{i,k}} \psi_{k} d\mu - \int_{M \setminus F_{i,k}} \psi_{k} d\mu+ \int_{M \setminus F_{i,k}} \psi_{k}\circ f^{i} d\mu = \\
 \int \psi_{k}d\mu + \int_{M \setminus F_{i,k}} \psi_{k}\circ f^{i} d\mu + \int_{M \setminus F_{i,k}} - \psi_{k} d\mu   =\\
  \int \psi_{k}d\mu + \int_{M \setminus F_{i,k}} \varphi_{1}\circ f^{i} d\mu + \int_{M \setminus F_{i,k}} - \psi_{k} d\mu \leq  \\
  \int \psi_{k}d\mu + \int_{M \setminus F_{i,k}} \beta d\mu + \int_{M \setminus F_{i,k}}  \beta d\mu \leq \\
\int \psi_{k}d\mu + 2 \beta \mu(M \setminus F_{i,k}).
\end{align*}

 since $-\beta \leq \psi_{k} \leq \max\{\varphi_{-} + \varepsilon,  \beta\} $. Then

\begin{center}
$ \int \psi_{k} \circ f^{i} d\mu \leq \int \psi_{k}d\mu +    2 \beta \mu(M \setminus F_{i,k})$,

\end{center}

we obtain that

\begin{equation}\label{eq4a}
 \frac{1}{n}\sum\limits_{i=0}^{n-k-1} \int \psi_{k} \circ f^{i} d\mu \leq   (1-\frac{k}{n}) \int \psi_{k}d\mu  + 2 \beta\cdot\frac{1}{n}\sum\limits_{i=0}^{n-k-1}  \mu(M \setminus F_{i,k}).
\end{equation}

By (\ref{ex1a}), (\ref{ex2a}), and the inequality above we have that

\begin{center}
$\frac{1}{n}\int\varphi_{n}(x) d\mu  \leq \frac{k}{n} (\int_{S} \varphi_{-}+\varepsilon d\mu + \beta) + (1-\frac{k}{n}) \int \psi_{k}d\mu+ 2 \beta\cdot\frac{1}{n}\sum\limits_{i=0}^{n-k-1}  \mu(M \setminus F_{i,k})$.
\end{center}

Passing $\limsup\limits_{n}$  in the previous inequality

\begin{align*}
 L= \limsup\limits_{n_{t}} \frac{1}{n_{t}}\int\varphi_{n_{t}}(x) d\mu    \leq \limsup\limits_{n} \frac{1}{n}\int\varphi_{n}(x) d\mu   \leq \\
\int \psi_{k}d\mu + 2\beta \limsup\limits_{n }\frac{1}{n}\sum\limits_{i=0}^{n-k-1}  \mu(f^{-i}(M \setminus E_{k}^{\varepsilon})).
\end{align*}

By equation (\ref{eq:int11a}),

\begin{align*}
  \int \varphi_{-}d\mu \leq L \leq \int \psi_{k}d\mu + 2\beta \limsup\limits_{n }\frac{1}{n}\sum\limits_{i=0}^{n-k-1}  \mu(f^{-i}(M \setminus E_{k}^{\varepsilon})).
 \end{align*}

Taking $\varepsilon = \frac{1}{\ell}$ for $\ell \in \mathbb{N}\setminus \{0\}$,

\begin{align*}
    \int \varphi_{-}d\mu \leq L \leq \lim\limits_{k \to \infty}\int \psi_{k}d\mu + 2\beta \lim\limits_{k \to \infty} \limsup\limits_{n }\frac{1}{n}\sum\limits_{i=0}^{n-k-1}  \mu(f^{-i}(M \setminus E_{k}^{\frac{1}{\ell}})).
 \end{align*}

By hypothesis $(b)$,

 \begin{align*}
 \int \varphi_{-}(x) d\mu\leq L \leq \lim \limits_{k \to \infty}   \int \psi_{k} d\mu.
 \end{align*}

\begin{lemma}
$\int \varphi_{-}d\mu = L$.
\end{lemma}
\begin{proof}

Using that $M = \bigcup\limits_{k=1}^{\infty} E_{k}^{\frac{1}{\ell}}$, we obtain that $\psi_{k} \to_{k} \varphi_{-} +\frac{1}{\ell}$ in each point. But \begin{center}
$-\beta \leq \psi_{k} \leq \max\{\varphi_{-} + \frac{1}{\ell},  \varphi_{1}^{+}\}$,
\end{center} we define $g:= \max\{\varphi_{-} + \frac{1}{\ell},  \beta\}$. So  $g$ is integrable and $|\psi_{k}| \leq g$. By dominated convergence theorem, we have that \begin{center}
$\lim\limits_{k} \int \psi_{k} d\mu = \int \varphi_{-} + \frac{1}{\ell} d\mu$.
\end{center}

We obtain that

\begin{center}
 $\int \varphi_{-}  d\mu \leq L \leq \lim \limits_{k}   \int \psi_{k} d\mu  =\int \varphi_{-}  d\mu + \frac{1}{\ell} $.
 \end{center} Making $\ell$ tend to infinity,

\begin{center}
$\int \varphi_{-}  d\mu \leq L \leq   \int \varphi_{-}  d\mu$
\end{center}
\end{proof}

Since $\int \varphi_{-}d\mu = L$ for all accumulation point $L$  of the sequence  $(\frac{1}{n}\int\varphi_{n}d\mu)_{n}$,  we have that $\lim\limits_{n} \frac{1}{n}\int\varphi_{n}d\mu = \inf\limits_{n} \frac{1}{n}\int\varphi_{n}d\mu =  \int \varphi_{-}d\mu$. This concludes the proof of Theorem \ref{mtheo11a}.

\end{proof}

Now, we are going to use a truncation argument to finish the proof of Theorem \ref{mtheo0a}. For each $k $ in $\mathbb{N}$ define $\varphi_{n}^{k}=\max\{\varphi_{n},-kn\}$ and $\varphi_{-}^{k}=\max\{\varphi_{-},-k\}$. For each  $\varepsilon >0$ fixed and $r \in \mathbb{N}$ we define $G_{r}^{\varepsilon}=\{x \in M:\varphi_{j}^{k}(x) \leq j (\varphi_{-}^{k}(x)+\varepsilon)$ for some  $j \in \{1,...,r\}\}$.

Using subadditivity of sequence $(\varphi_{n} )_{n}$ and definitions, it is possible to check the following Lemma.

\begin{lemma}\label{lemma-assumpa}

\begin{itemize}
\item [$(i)$] $(\varphi_{n}^{k})_{n}$ is a subadditive sequence for any $k$ fixed.
\item [$(ii)$] $\varphi_{1}^{k}$ is upper bounded for any $k$ fixed.
\item [$(iii)$] $(\frac{\varphi_{n}^{k}}{n})_{n}$ is uniformly bounded by below for any $k$ fixed.
\item [$(iv)$] $\varphi_{-}^{k}(x) = \liminf\limits_{n} \frac{\varphi_{n}^{k}(x)}{n}$ for any $k$ fixed.
\item[$(v)$] For each $j \in \mathbb{N}$ we have that $\varphi_{-}^{k}(f^{j}(x)) = \varphi_{-}^{k}(x)$  $\mu-a.e.$ $x$ in $M$ where $\varphi_{-}^{k}: M \to [-\infty, \infty]$ is given by  $\varphi_{-}^{k}(x) = \liminf\limits_{n} \frac{\varphi_{n}^{k}(x)}{n}$ for any $k$ fixed.
\item[$(vi)$]  $E_{r}^{\varepsilon} \subseteq G_{r}^{\varepsilon}$ for every $\varepsilon > 0$ and $r \in \mathbb{N}$.

\item [$(vii)$]  Fixed $n$,  $(\varphi_{n}^{k})_{k}$ is a nonincreasing  monotonic sequence.
\item [$(viii)$] Fixed $n$, $\lim\limits_{k} \varphi_{n}^{k}(x) = \varphi_{n}(x)$ for all $x$ in $M$, (then $\varphi_{n}^{k} \searrow_{k} \varphi_{n}$).

\item [$(ix)$]  $(\varphi_{-}^{k})_{k}$ is a nonincreasing  monotonic sequence.
\item [$(x)$] $\lim\limits_{k} \varphi_{-}^{k}(x) = \varphi_{-}(x)$ for all $x$ in $M$, (then $\varphi_{-}^{k} \searrow_{k} \varphi_{-}$).
\item [$(xi)$] $(\varphi_{-}^{k})^{+}(x) =(\varphi_{-})^{+}(x)$ for all  $x$ in $M$ and for all $k$ in $\mathbb{N}$.
\end{itemize}
\end{lemma}
\begin{flushright}
$\blacksquare$
\end{flushright}

By $(vi)$,  $E_{r}^{\varepsilon} \subseteq G_{r}^{\varepsilon}$ for every $\varepsilon > 0$ and $r \in \mathbb{N}$. In particular,   $f^{-i}(M\setminus G_{r}^{\varepsilon}) \subseteq f^{-i}(M\setminus E_{r}^{\varepsilon})$ for all $i \geq 0$. Note that

\begin{center}
$\mu(f^{-i}(M\setminus G_{r}^{\varepsilon}))\leq \mu(f^{-i}(M\setminus E_{r}^{\varepsilon})) $, and then
\end{center}

 $\lim\limits_{r \to +\infty} \limsup \limits_{n} \frac{1}{n}\sum \limits_{i=0}^{n-r-1} \mu(M\setminus f^{-i}(G_{r}^{\varepsilon})) \leq \lim\limits_{r \to +\infty} \limsup \limits_{n} \frac{1}{n}\sum \limits_{i=0}^{n-r-1} \mu(M\setminus f^{-i}(E_{r}^{\varepsilon}))$. Therefore for each $k$ we have that the sequence $(\varphi_{n}^{k})_{n}$ satisfies the hypothesis of Theorem \ref{mtheo11a}, so

\begin{equation}\label{r2a}
\int \liminf\limits_{n} \frac{\varphi_{n}^{k}(x)}{n} d\mu = \int \varphi_{-}^{k}d\mu =  \lim\limits_{n} \int \frac{\varphi_{n}^{k}}{n} d\mu = \inf\limits_{n} \int \frac{\varphi_{n}^{k}}{n} d\mu.
\end{equation}

We claim that
\begin{equation}\label{r3a}
\inf_{k} \int \varphi_{n}^{k}d\mu = \int \varphi_{n} d\mu.
\end{equation}

To see this recall that $\varphi_{n}^{k} \searrow_{k} \varphi_{n}$ with $\varphi_{n}^{k}=\max\{\varphi_{n},-kn\}$, so  $\varphi_{n}^{1} \geq  \varphi_{n}^{k}$  for all  $k$. Consider  $\gamma_{k} = \varphi_{n}^{1} - \varphi_{n}^{k} \geq 0$, and note that $\gamma_{k} = \varphi_{n}^{1} - \varphi_{n}^{k} \leq \varphi_{n}^{1} - \varphi_{n}^{k+1} = \gamma_{k+1}$. Thus $(\gamma_{k})_{k}$ is  nondecreasing  monotonic sequence and $\gamma_{k} \nearrow_{k} \varphi_{n}^{1} - \varphi_{n}$ , and by monotone convergence theorem, $\int  \varphi_{n}^{1} - \varphi_{n} d\mu  =\int \lim\limits_{k\to \infty} \gamma_{k} d\mu = \lim\limits_{k\to \infty} \int \gamma_{k} d\mu  = \lim\limits_{k\to \infty} \int \varphi_{n}^{1} - \varphi_{n}^{k} d\mu $, and then  $\int   \varphi_{n} d\mu = \lim\limits_{k\to \infty} \int  \varphi_{n}^{k} d\mu$ and $\lim\limits_{k\to \infty} \int  \varphi_{n}^{k} d\mu =\inf_{k} \int \varphi_{n}^{k}d\mu$

Similarly, using monotone convergence theorem,
\begin{equation}\label{r4a}
\inf\limits_{k} \int \varphi_{-}^{k}d\mu = \int \varphi_{-} d\mu
\end{equation}

By $(\ref{r2a})$, $(\ref{r3a})$ and $(\ref{r4a})$, we have that

\begin{center}
$\int \varphi_{-} d\mu=\inf\limits_{k} (\int \varphi_{-}^{k} d\mu) =  \inf\limits_{k} (\inf\limits_{n} \int \frac{\varphi_{n}^{k}}{n} d\mu) = \inf\limits_{n}\frac{1}{n} (\inf\limits_{k} \int \varphi_{n}^{k} d\mu)= \inf\limits_{n}\frac{1}{n} (\int \varphi_{n} d\mu) $
\end{center}

Then

\begin{equation}
\int \varphi_{-} d\mu = \inf\limits_{n}\frac{1}{n} \int \varphi_{n} d\mu.
\end{equation}

This concludes the proof  of Theorem \ref{mtheo0a}.


\def\cprime{$'$}


\end{document}